\documentclass[11pt]{article}
\usepackage[final]{epsfig}
\usepackage[left=3cm,right=3cm, top=2.5cm,bottom=2.5cm,bindingoffset=0cm]{geometry}
\usepackage{graphics}
\usepackage{amsmath}
\usepackage{amsfonts}
\usepackage{latexsym}
\usepackage{amssymb, mathabx}
\usepackage{amsthm}
\usepackage{graphicx}
\usepackage{epstopdf}
\usepackage{multicol,multirow}
\usepackage{hyperref, enumitem}
\usepackage{mathtools}

 \usepackage{relsize}

\usepackage[nottoc]{tocbibind}

\usepackage{amsthm}

\makeatletter
\renewenvironment{proof}[1][\proofname]{%
  \par\pushQED{\qed}%
  \normalfont \topsep6\p@\@plus6\p@\relax
  \trivlist
  \item[\hskip\labelsep\bfseries #1\@addpunct{.}]%
}{%
  \popQED\endtrivlist\@endpefalse
}
\makeatother

\mathtoolsset{showonlyrefs}

\usepackage{tikz}
\usepackage{tikz-cd} 
\usetikzlibrary{positioning}
\usetikzlibrary{decorations.text}
\usetikzlibrary{decorations.pathmorphing}
\usetikzlibrary{decorations.markings}

\usetikzlibrary{arrows} 
\tikzset{
    >=stealth',
    pil/.style={
           ->,
           thick,
           shorten <=2pt,
           shorten >=2pt,}
}
\tikzset{->-/.style={decoration={
  markings,
  mark=at position .7 with {\arrow{>}}},postaction={decorate}}}
  \tikzset{a/.style={decoration={
  markings,
  mark=at position .52 with {\arrow{angle 90}}},postaction={decorate}}}
\tikzset{-<-/.style={decoration={
  markings,
  mark=at position .4 with {\arrow{<}}},postaction={decorate}}}

\makeatletter
\def\thmhead@plain#1#2#3{%
  \thmname{#1}\thmnumber{\@ifnotempty{#1}{ }\@upn{#2}}%
  \thmnote{ {\the\thm@notefont#3}}}
\let\thmhead\thmhead@plain

\makeatother

\newcounter{AppCounter}

\def\restrict#1{\raise-.5ex\hbox{\ensuremath|}_{#1}}

\newtheorem{lemma}{Lemma}[section]
\newtheorem{proposition}[lemma]{Proposition}

\newtheorem{remark-definition}[lemma]{Remark-Definition}
\newtheorem{theorem}[lemma]{Theorem}
\newtheorem{corollary}[lemma]{Corollary}

\newtheorem{proposition-conjecture}[lemma]{Proposition-conjecture}

\theoremstyle{definition}
\newtheorem{example}[lemma]{Example}

\newtheorem{definition}[lemma]{Definition}
\newtheorem{remark}[lemma]{Remark}

\newcommand{\R}{{\mathbb R}}

\newcommand{\B}{  B}

\newcommand{\BVect}{\mathrm{BVect}}

\newcommand{\Ker}{\mathrm{ker}}

\newcommand{\T}{{T}}
\newcommand{\Cont}{{C}}
\newcommand{\LieBracket}{ [\, , ] }

\newcommand{\id}{\mathrm{id}}

\newcommand{\Vect}{\mathrm{Vect}}
\newcommand{\Diff}{\textnormal{Diff}} 

\newcommand{\Src}{\mathrm{src}}
\newcommand{\Trg}{\mathrm{trg}}

\newcommand{\A}{{A}}

\newcommand{\low}[1]{\raise-.0ex\hbox{$\scriptstyle #1$}}
\newcommand{\high}[1]{\raise.5ex\hbox{$\scriptstyle #1$}}

\newcommand{\dbo}{\left[\!\left[}
\newcommand{\dbc}{\right]\!\right]}

\renewcommand{\tfrac}{\frac}

\newcommand{\marginnote}[1]
{
}

\newcounter{bk}

\title {Virasoro extensions for diffeomorphisms with breaks}

\author{Anton Izosimov\thanks{
School of Mathematics and Statistics,
University of Glasgow;
e-mail: {\tt anton.izosimov@glasgow.ac.uk}
},
Boris Khesin\thanks{
Department of Mathematics,
University of Toronto;
e-mail: \tt{khesin@math.toronto.edu}
}, 
and Howard Xiao\thanks{Department of Electrical Engineering, Stanford University; e-mail: \tt{howardx@stanford.edu}}
}

\date{}

\begin{document}

\maketitle
\begin{abstract}
We study homeomorphisms of the circle that are smooth diffeomorphisms away from a finite set of 
$
n$ points. These “broken diffeomorphisms” do not form a Lie group, but instead naturally assemble into a Lie groupoid. We construct an explicit nontrivial 
$n$-dimensional central extension of this groupoid, which restricts to the classical Virasoro group when confined to smooth diffeomorphisms. We further describe the associated “broken Virasoro” algebroid, defined as a nontrivial  $n$-dimensional central extension of the Lie algebroid of vector fields on the circle that are smooth except at 
$n$ points. This construction generalizes the Virasoro algebra.

As a byproduct, we analyze a related setting on an interval: we construct a nontrivial central extension of the Lie algebra of vector fields vanishing at the endpoints, together with the corresponding central extension of the group of diffeomorphisms fixing the endpoints. We also describe the associated Lie algebroid and groupoid obtained by allowing the endpoints to vary. 

\end{abstract}

\tableofcontents

\section{Introduction} \label{intro}

The Virasoro algebra and its associated group are among the most remarkable structures in modern mathematical physics, arising naturally in conformal field theory, string theory, integrable systems, and many other contexts. 

The \emph{Virasoro algebra} is the unique nontrivial central extension of the Lie algebra of vector fields $u(x)\partial_x$ on the circle $S^1$ defined by the \emph{Gelfand-Fuchs 2-cocycle}
\[
\Omega(u\partial_x, v\partial_x) = \int_{S^1} u\, v_{xxx}\, dx.
\]
The corresponding \emph{Virasoro-Bott group} is the unique extension of the group of circle diffeomorphisms  $x \mapsto \phi(x)$ defined by the \emph{Bott 2-cocycle}
\[
\chi(\phi, \psi) = \int_{S^1} \log\big((\phi \circ \psi)_x\big)\, d\log(\psi_x).
\]

In this paper, we consider homeomorphisms of the circle that are smooth diffeomorphisms away from a finite set of points $p_1, \dots, p_n$. Such maps are well known in dynamical systems, where they, in particular, arise in the context of renormalization~\cite{Khanin2022}. The goal of this paper is to define analogues of the Virasoro group and algebra in this setting. On the infinitesimal level, homeomorphisms that are smooth diffeomorphisms away from \(n\) points correspond to continuous vector fields that are smooth away from \(n\) points. The space of such vector fields is not a Lie algebra, as it is not closed under the usual Lie bracket. Indeed, the bracket of two such fields may fail to be continuous. To address this, we equip the space \(\BVect_n(S^1)\) of continuous vector fields on \(S^1\) that are smooth away from \(n\) points with the more general structure of a \emph{Lie algebroid}. This algebroid is a natural generalization of the Lie algebra of smooth vector fields on the circle. Moreover, this Lie algebroid admits a nontrivial \(2\)-cocycle, which can be viewed as an algebroid analogue of the Gelfand--Fuchs cocycle.


\begin{theorem} ($=$Theorem \ref{thm:GFA})
The algebroid \(\BVect_n(S^1)\) of continuous vector fields on \(S^1\) that are smooth away from \(n\) points admits \(n\) nontrivial \(2\)-cocycles, given by a variant of the Gelfand--Fuchs cocycle:
\[
\Omega_i(u\partial_x, v\partial_x)
= \int_{p_i}^{p_{i+1}} \bigl(u_x v_{xx} - u_{xx} v_x\bigr)\,dx,
\]
where \(p_1, \dots, p_n\) are the points at which the vector fields \(u\partial_x\) and \(v\partial_x\) fail to be smooth.
\end{theorem}
These \(n\) cocycles on \(\BVect_n(S^1)\) define a nontrivial \(n\)-dimensional central extension, which we call the \emph{Virasoro algebroid}.

Furthermore, we show that the set of circle homeomorphisms that are diffeomorphisms away from \(n\) points admits a natural structure of a Lie groupoid integrating the Lie algebroid \(\BVect_n(S^1)\). This groupoid, denoted \(\mathrm{BDiff}(S^1,n)\), also admits a central extension. The corresponding cocycles integrate the algebroid versions of the Gelfand-Fuchs cocycle and can be viewed as groupoid counterparts of the Bott cocycle.


\begin{theorem}(\(=\) Theorem~\ref{thm:vir-groupoid})
The groupoid \(\mathrm{BDiff}(S^1,n)\) of broken circle diffeomorphisms admits a nontrivial \(n\)-dimensional central extension given by a variant of the Bott cocycle,
\[
\chi_i(\phi, \psi)
= \int_{p_i}^{p_{i+1}} \log\bigl(\phi_x \circ \psi\bigr)\, d\log(\psi_x),
\]
where \(p_1, \dots, p_n\) are the break points.
\end{theorem}

Note that in both the algebroid and groupoid settings, the corresponding cocycles must be written in the above special form, unlike the many equivalent expressions available in the smooth setting. For example, the expressions
\[
\int u\, v_{xxx}\, dx
\quad \text{and} \quad
\int \bigl(u_x v_{xx} - u_{xx} v_x\bigr)\,dx
\]
differ only by a constant factor in the smooth case. However, only the latter expression defines a cocycle in the setting with break points. Likewise, in the smooth group setting, both formulas
\[
\int \log\bigl((\phi \circ \psi)_x\bigr)\, d\log(\psi_x) \quad \text{and} \quad
\int \log\bigl(\phi_x \circ \psi\bigr)\, d\log(\psi_x), 
\]
can be used to define a cocycle, but only the latter formulation remains valid in the presence of break points.

Next, we show how these results shed a new light on the more familiar objects, the Lie group of diffeomorphisms of an interval, and its Lie algebra of vector fields on an interval vanishing at the endpoints.
(The latter Lie algebra appears in the context of the anchor map for the above Virasoro algebroid.)
We describe  central extensions of those Lie group of diffeomorphisms and Lie algebra of vector fields related to the interval, see Section \ref{sec:avatars}:

\begin{theorem}
The  diffeomorphism group $\mathrm{Diff}(I)$ for a segment $I=[a,b]$ has a nontrivial one-dimensional central extension given by the following group 2-cocycle $\chi$: for any pair of interval diffeomorphisms $\phi, \psi$
\[
\chi(\phi, \psi) = \int_{a}^{b} \log(\phi_x \circ \psi)\, d\log (\psi_x)\,.
\]
The corresponding Lie algebra $\mathrm{Vect}_0(I)$ of vector fields vanishing 
at the endpoints has a one-dimensional extension 
by means of the Lie algebra 2-cocycle
\[
\Omega\!\left(u(x)\partial_x,\, v(x)\partial_x\right)
:=
\int_{a}^{b} \bigl(u_x v_{xx} - u_{xx} v_x\bigr)\,dx
\]
for a pair of vector fields $u(x)\partial_x,\, v(x)\partial_x$.
\end{theorem}

The nontriviality of this cocycle is proved in Theorem~\ref{NCB}.
One can also generalize those cocycles to groupoids of diffeomorphisms of intervals with moving ends and  to the corresponding algebroids, see Proposition~\ref{prop:interval}.

\medskip

Returning to diffeomorphisms with a finite number of breaks, their set can be regarded as the configuration space for a one-dimensional compressible fluid with vortex sheets,  cf.~\cite{IzoKh2017}.
Furthermore, the Virasoro group is known to be related to the classification of Hill's operators and projective structures, while the Korteweg-de Vries and Camassa-Holm equations on the circle appear as the Euler-Arnold equations on the Virasoro algebra \cite{KhMis2003, ovsienko2005projective}. In the groupoid-algebroid setting this leads to the corresponding classification of symplectic leaves in the Poisson bundle dual to the broken Virasoro algebroid and to the intriguing appearance of boundary terms. We are going to address it in a future publication.

Another direction is related to the composition of orientation-preserving diffeomorphisms 
with breaks on the circle. Note that if one does not match the break points of one broken diffeomorphism with those of the next one, the number of breaks naturally increases. By taking the union of all such images, one obtains the group of all circle diffeomorphisms with a finite number of breaks; this forms a subgroup (dense in the $C^0$-topology) of the group $Homeo(S^1)$ of circle homeomorphisms.  An ingenious construction of the Lie algebra corresponding to the group $Homeo(S^1)$ was presented in \cite{MalikovPenner1998}. This Lie algebra turns out to be the algebra of ``piecewise $\mathfrak{sl}_2$'' vector fields on the circle. The paper also raises the question of a possible relation to the Virasoro algebra, and hence, implicitly, of possible central extensions of their construction. One may hope that the groupoid approach presented below, with central extensions whose dimension is determined by the number of breaks, could provide a partial answer to this question. 

\bigskip

\textbf{Acknowledgments.} A.I. is grateful to Max Planck Institute for Mathematics in Bonn for
its hospitality and financial support. A.I. was partially supported by the Simons Foundation
through its Travel Support for Mathematicians program. The research of BK was partially supported by an NSERC Discovery Grant. 

\section{The Gelfand-Fuchs cocycle on an interval}\label{sec:avatars}

\subsection{Avatars of the Virasoro algebra}
Consider the Lie algebra \(\mathrm{Vect}(S^1)\) of smooth vector fields on the circle. Upon choosing a coordinate \(x\), its elements can be written as \(u(x)\partial_x\), where \(u \in C^\infty(S^1)\) is a smooth function.

\begin{definition}\label{def:GF-cocycle}
The Virasoro Lie algebra \(\mathrm{vir} := \mathrm{Vect}(S^1) \oplus \mathbb{R}\) is the central extension of \(\mathrm{Vect}(S^1)\) defined by the Gelfand--Fuchs cocycle \(\Omega\), which can be written in either of the following equivalent forms:
\[
\Omega(u\partial_x, v\partial_x)
= \int_{S^1} u\, v_{xxx}\, dx
= -\int_{S^1} u_x v_{xx}\, dx,
\]
where \(u(x)\partial_x\) and \(v(x)\partial_x\) are smooth vector fields on \(S^1\), see \cite{Fuchs}. 
\end{definition}
The commutator in the Virasoro algebra is given by
\[
[(u\partial_x, \alpha), (v\partial_x, \beta)]
=
\bigl((uv_x - u_x v)\partial_x,\; \Omega(u\partial_x, v\partial_x)\bigr),
\]
for any \((u\partial_x, \alpha), (v\partial_x, \beta) \in \mathrm{Vect}(S^1) \oplus \mathbb{R}\). This extension is known to be nontrivial, i.e. non-isomorphic to the direct sum \(\mathrm{Vect}(S^1) \oplus \mathbb{R}\) as a Lie algebra. Equivalently, the Gelfand--Fuchs cocycle \(\Omega\) is not a coboundary. In other words, it represents a nontrivial cohomology class in \(H^2(\mathrm{Vect}(S^1), \mathbb{R})\); see, for example, \cite{Segal}.

Now consider the Lie algebra \(\mathrm{Vect}(I)\) of smooth vector fields on an interval \(I = [a,b]\). We will be interested in its subalgebra \(\mathrm{Vect}_0(I)\) consisting of vector fields that vanish at the endpoints \(a\) and \(b\). It is straightforward to verify that this is indeed a Lie subalgebra, i.e. it is closed under the Lie bracket. Our interest in \(\mathrm{Vect}_0(I)\) stems from the fact that, unlike the full Lie algebra \(\mathrm{Vect}(I)\), it integrates to a Lie group of orientation-preserving diffeomorphisms of the interval \(I\).

It turns out that a skew-symmetrized version of the Gelfand--Fuchs cocycle allows one to define a central extension of \(\mathrm{Vect}_0(I)\):

\begin{definition}\label{def:interval-coc}
 The extended Lie algebra $ \mathrm{Vect}_0(I) \oplus \mathbb R$ of vector fields on an interval $I = [a,b]$ vanishing at the endpoints $a$ and $b$ is the central extension of 
 $\mathrm{Vect}_0(I)$ defined by the 2-cocycle
 \[
\Omega\!\left(u(x)\partial_x,\, v(x)\partial_x\right)
\;:=\;
\int_a^b \bigl(u_x v_{xx} - u_{xx} v_x\bigr)\,dx
\]
for any pair of vector fields $u\partial_x, v\partial_x \in \mathrm{Vect}_0(I)$.
\end{definition}

Note that, in order to define \(\Omega\), we use this specific form of the cocycle. While there are many equivalent expressions in the case of the circle, they may differ by boundary terms when restricted to an interval.  One can verify directly that the above skew-symmetric form of \(\Omega\) satisfies the cocycle identity. We omit this computation here, as it will be carried out later {in two different ways: in the more general setting of algebroids and in the context of differential operators on an interval.} It turns out that this indeed defines a nontrivial central extension of the Lie algebra \(\mathrm{Vect}_0(I)\):

\begin{theorem}\label{NCB}
    The cocycle $\Omega$ is not a coboundary.
\end{theorem}
    
\begin{proof}
Without loss of generality, assume \(a = 0\) and \(b = \pi\), and set \(e_n := \sin(nx)\,\partial_x\). Then
\[
[e_m, e_n] = \frac{n-m}{2}e_{n+m} + \frac{m+n}{2}e_{m-n}.
\]
Moreover,
\[
\Omega(e_m, e_n) =
\begin{cases}
\displaystyle \frac{2mn(m^2 + n^2)}{m^2 - n^2}, & \text{if \(m,n\) have different parity},\\[6pt]
0, & \text{otherwise}.
\end{cases}
\]

Suppose, for contradiction, that \(\Omega\) is a coboundary. By definition, this means that there exists a linear functional \(\lambda\) on the Lie algebra \(\mathrm{Vect}_0(I)\) such that
\[
\Omega\!\left(u(x)\partial_x,\, v(x)\partial_x\right)
= \lambda\left(\left[u(x)\partial_x,\, v(x)\partial_x\right]\right).
\]
In particular,
\begin{equation}
\Omega(e_m, e_n) = \lambda\left(\left[e_m, e_n\right]\right).
\end{equation}
Setting \(\lambda_k := \tfrac{1}{2}\lambda(e_k)\), we obtain
\[
\Omega(e_m,e_n) = (n-m)\lambda_{m+n} + (m+n)\lambda_{m-n}.
\]
Thus,
\[
\frac{2mn(m^2 + n^2)}{m^2 - n^2}
= (n-m)\lambda_{m+n} + (m+n)\lambda_{m-n}
\]
whenever \(m\) and \(n\) have different parity. Setting \(k = m+n\) and \(l = m-n\), this relation becomes
\[
\frac{k^4 - l^4}{4kl} = k \lambda_l - l \lambda_k,
\]
for all \(k,l \in 2\mathbb{Z}+1\). Note that these equations are invariant under the substitution \(\lambda_n \mapsto \lambda_n - n\lambda_1\). Indeed, \(\lambda_n = n\) solves the corresponding homogeneous equation with vanishing left-hand side. Thus, without loss of generality, we may assume \(\lambda_1 = 0\). Setting \(l = 1\), we obtain
\[
\frac{k^4 - 1}{4k} = k \lambda_1 - \lambda_k = - \lambda_k,
\]
and hence
\[
\lambda_k = - \frac{k^4 - 1}{4k}.
\]
Substituting this expression back into the original relation, one checks that
\[
\frac{k^4 - l^4}{4kl} = k \lambda_l - l \lambda_k
\]
fails to hold for general \(k,l\). This contradiction shows that no such sequence \(\{\lambda_k\}\) exists. Therefore, \(\Omega\) is not a coboundary.
\end{proof}


\subsection{Avatars of the Virasoro-Bott group}
There is an infinite-dimensional Lie group whose Lie algebra is the Virasoro algebra.

\begin{definition}
     The Virasoro-Bott group  $\mathrm{VB}(S^1):=\mathrm{Diff}(S^1)\oplus \mathbb R$
    is the central extension of the diffeomorphism group   $\mathrm{Diff}(S^1)$ of the circle 
    by means of the     Bott group cocycle $\chi(\phi, \psi)$ of the following  form:
\[
\chi(\phi, \psi)=\int_{S^1} \log (\phi\circ \psi)_x \, d\log (\psi_x)\,,
\]
where $\phi, \psi$ are diffeomorphisms of the circle $S^1$.
\end{definition}
The group multiplication in the Virasoro-Bott group $VB(S^1)$ is given by
\[
(\phi, s)\circ (\psi, t):=(\phi\circ \psi, s+t+ \chi(\phi, \psi) )
\]
for any pair $(\phi, s), (\psi, t)\in \mathrm{Diff}(S^1)\oplus \mathbb R$. The cocycle identity
$$
B(\phi\circ\psi, \eta)+B(\phi, \psi)=B(\phi, \psi\circ \eta)+B(\psi, \eta)
$$
for any triple of diffeomorphisms $\phi,\psi, \eta$ ensures that the multiplication for the group's extension is associative. It is easy to see that the infinitesimal version of the cocycle $\chi$
is the Gelfand-Fuchs 2-cocycle defining the Virasoro algebra. In particular, this group 
extension $\mathrm{VB}(S^1)$ is nontrivial since its infinitesimal Lie-algebraic cocycle is nontrivial.

It turns out that the Lie group $\mathrm{Diff}(I)$ of orientation-preserving diffeomorphisms of an interval $I=[a,b]$ also has a nontrivial central extension with the following explicit description.

\begin{definition}\label{def:gc}
     The Virasoro-Bott-type group  $\mathrm{VB}(I):=\mathrm{Diff}(I)\oplus \mathbb R$ for an interval
    $I=[a,b]$ is the central extension of the orientation-preserving diffeomorphism group   $\mathrm{Diff}(I)$ of the interval  by means of the  following group cocycle $\chi(\phi, \psi)$:
\[
\chi(\phi, \psi)=\int_a^b \log (\phi_x\circ \psi) \, d\log (\psi_x)\,,
\]
where $\phi, \psi$ are diffeomorphisms of the circle $S^1$.
\end{definition}
Again, note that we use a specific form of the cocycle on the interval.  It is equivalent to the Bott cocycle on the circle, but differs from the latter by a boundary contribution when restricted to an interval.

We will verify the group cocycle identity in the more general setting of the groupoid cocycle in Section~\ref{sect:brok-diff}. Similarly to the circle case, the group extension \(\mathrm{VB}(I)\) is nontrivial, since its infinitesimal Lie-algebraic cocycle is nontrivial.

\begin{remark}
One can adapt the above definitions to the Lie algebra of vector fields with compact support or rapid decay on the real line \(\mathbb{R}\), where the usual cocycle formula involving third derivatives applies, as well as to a half-line \([a, \infty)\), where it does not. In the latter case, the cocycle is only defined on vector fields vanishing at the endpoint \(a\) of the half-line. In both the line and half-line cases, to define the group cocycle one considers only diffeomorphisms that are equal to, or rapidly approach, the identity at infinity.
\end{remark}

\begin{remark} 
Another variation of the above consideration, relevant for what follows, is the Lie algebra
of vector fields on a finite collection of pairwise adjacent intervals, or equivalently,
the Lie algebra $\mathrm{Vect}_0^n(I)$ of vector fields on the interval $[a,b]$ with any partition $a=p_0<p_1, \dots <p_n=b$ 
vanishing at all intermediate and endpoints 
$p_j, j=0,\dots, n$, but not necessarily differentiable at those points. The above consideration
implies the following.
\end{remark}

\begin{corollary}
The Lie algebra \(\mathrm{Vect}_0^n(I)\) admits an \(n\)-dimensional nontrivial central extension, given by an arbitrary linear combination of the cocycles \(\sum_{j} \lambda_j \Omega_j\) with \(\lambda_j \in \mathbb{R}\), where \(\Omega_j\) is defined as in Definition~\ref{def:interval-coc}, but with integration over the interval \([p_j, p_{j+1}]\).

Similarly, the group \(\mathrm{Diff}^n(I)\) of homeomorphisms of the interval \([a,b]\) fixing the points \(a = p_0 < p_1 < \cdots < p_n = b\) and restricting to diffeomorphisms away from these points admits an \(n\)-dimensional nontrivial central extension. This extension is given by the group cocycle from Definition~\ref{def:gc}, with integration over each interval \([p_j, p_{j+1}]\).
\end{corollary}

\medskip


\section{Lie algebroids and their central extensions}\label{sect:alg-ext}

\subsection{Lie algebroids}

\begin{definition}
A \textit{Lie algebroid} $\A$ over a manifold $\B$ (called \emph{the base} of the algebroid) is a vector bundle $\A \to \B$ endowed with a Lie bracket $\LieBracket$ on $\Cont^\infty$-smooth sections and a vector bundle morphism $\# \colon \A \to \T \B$, called the \textit{anchor map}, such that for any two $\Cont^\infty$-sections $\mathbf{u}, \mathbf{v} \in \Gamma(\A)$ and any smooth function $f \in \Cont^\infty(\B)$, one has the following version of the Leibniz rule:
$$
[\mathbf{u},f\mathbf{v}] = f[\mathbf{u},\mathbf{v}] + (L_{\#\mathbf{u}}  f) \mathbf{v}\,.
$$
\end{definition}

\begin{remark}
Here and in what follows $\Gamma(E)$ stands for the space of smooth sections of a vector bundle $E$, while $ L_{v} f$ stands for the derivative of the function $f$ along the vector 
field $v$ on $\B$.  
\end{remark}

\begin{example}\label{algebroids} The following are examples of Lie algebroids:
\begin{enumerate}[label=(\alph*)] 
\item
A Lie algebra, considered as a Lie algebroid over a point.
\item The tangent bundle $\T \B$ of the manifold $\B$. The bracket on sections is the standard Lie bracket of vector fields, while the anchor map is the identity.
\end{enumerate}
\end{example}

\subsection{Central extensions and cocycles}

The theory of central extensions for Lie algebroids is parallel to that for Lie algebras. In particular, a $1$-dimensional central extension of an algebroid $\A \to \B$ is defined as follows. As a vector bundle, it is a direct sum of $\A$ and a trivial line bundle over $\B$. 
The anchor map on this extended bundle $\hat A \to B$ is obtained by composing the anchor of $A$ with the natural projection $\hat A \to A$. Sections of $\hat A$ can be identified with pairs $(\mathbf{u}, f)$, where $\mathbf{u} \in \Gamma(A)$ and $f \in C^\infty(B)$ is a section of the trivial line bundle. The bracket on sections is defined by
\[
[(\mathbf{u}, f), (\mathbf{v}, g)]
=
\big([\mathbf{u}, \mathbf{v}],\, \Omega(\mathbf{u}, \mathbf{v}) + L_{\#\mathbf{u}} g - L_{\#\mathbf{v}} f \big).
\]
where $\Omega \in \Gamma(\Lambda^2 A^*)$ is a skew-symmetric 2-form on $A$. This bracket always satisfies the Leibniz rule, while the Jacobi identity holds if and only if
\[
L_{\#\mathbf{u}} \Omega(\mathbf{v}, \mathbf{w}) - \Omega([\mathbf{u}, \mathbf{v}], \mathbf{w}) \,+ \circlearrowleft \,= 0,
\]
for all $\mathbf{u}, \mathbf{v}, \mathbf{w} \in \Gamma(A)$, where $\circlearrowleft$ denotes cyclic permutations. The left-hand side is $C^\infty(B)$-linear in $\mathbf{u}, \mathbf{v}, \mathbf{w}$ and thus defines a 3-form $d\Omega \in \Gamma(\Lambda^3 A^*)$, known as the \emph{Lie algebroid differential} of $\Omega$. A 2-form $\Omega$ is called a \emph{2-cocycle} if $d\Omega = 0$. Hence, a one-dimensional central extension of a Lie algebroid $A$ is completely determined by a 2-cocycle on $A$.

More generally, one can define a sequence of maps
\[
d_k : \Gamma(\Lambda^k A^*) \to \Gamma(\Lambda^{k+1} A^*), 
\qquad \text{with } d_{k+1} \circ d_k = 0,
\]
forming the \emph{Chevalley--Eilenberg complex} of the algebroid. The 2-cocycles determining central extensions lie in $\ker d_2$. Any element of the image of $d_1$ is automatically a 2-cocycle; such cocycles are called \emph{coboundaries}. Explicitly, if $\Theta \in \Gamma(A^*)$ is a 1-form, its differential
\[
d\Theta (\mathbf{u}, \mathbf{v}) = L_{\#\mathbf{u}} \Theta(\mathbf{v}) - L_{\#\mathbf{v}} \Theta(\mathbf{u}) - \Theta([\mathbf{u}, \mathbf{v}])
\]
defines a 2-coboundary. Cocycles that are coboundaries correspond to \emph{trivial central extensions}, i.e., those isomorphic to the direct sum of $A$ and a trivial line bundle endowed with zero bracket and anchor. More generally, central extensions up to natural  equivalence  
are classified by the quotient
\[
H^2(A) = \frac{\ker d_2}{\operatorname{im} d_1},
\]
the \emph{second Lie algebroid cohomology} of $A$.

\begin{example}
In the case of a Lie algebra, the cocycle condition becomes the standard condition for a Lie algebra $2$-cocycle:
$$
\Omega([\mathbf{u}, \mathbf{v}], \mathbf{w}) \,+ \circlearrowleft \,= 0.
$$
\end{example}

\begin{example}
When $\A = T\B$ is the tangent bundle, the Lie algebroid differential is just the de Rham differential, so a central extension is determined by a closed $2$-form on $\B$. 
\end{example}
Analogously, one defines $n$-dimensional central extensions. In that case, the cocycle $\Omega$ takes values in $\R^n$, which is equivalent to specifying $n$ cocycles $\Omega_1, \dots, \Omega_n$.

\subsection{Algebroid and algebra cocycles}

A cocycle on a Lie algebroid also induces a cocycle on each of its \textit{isotropy algebras}, as follows.

Let $x \in B$, and let $u, v \in \Ker \#_x$ be elements of the kernel of the anchor map restricted to the fiber of $\A$ over $x$. Let $\mathbf{u}, \mathbf{v}$ be arbitrary smooth sections of $\A$ such that $\mathbf{u}(x) = u$ and $\mathbf{v}(x) = v$. Then one easily checks the following:

\begin{proposition}\label{prop:isoalgebra}
The value $[\mathbf{u}, \mathbf{v}](x)$ depends only on $u$ and $v$, but not on the choice of the extensions $\mathbf{u}, \mathbf{v}$. Therefore, the formula
\[
[u, v] := [\mathbf{u}, \mathbf{v}](x)
\]
defines a well-defined bracket on $\Ker \#_x$. This bracket endows $\Ker \#_x$ with the structure of a Lie algebra. 
\end{proposition}

\begin{definition}
The Lie algebra $\Ker \#_x$ is called the \textit{isotropy algebra} at the point $x$.
\end{definition}

\begin{proposition}
Suppose $\Omega \in \Gamma(\Lambda^2 A^*)$ is a 2-cocycle on the Lie algebroid $A$. Then, for any $x \in B$, the restriction of $\Omega$ to the isotropy algebra $\ker \#_x$ is a Lie algebra 2-cocycle. Moreover, if $\Omega$ is a coboundary, then its restriction to $\ker \#_x$ is also a coboundary.
\end{proposition}

\begin{proof}
Restriction to the isotropy algebra at any point defines a cochain map from the Chevalley--Eilenberg complex of the Lie algebroid $A$ to the Chevalley--Eilenberg complex of the Lie algebra $\ker \#_x$ (this follows, e.g., from \cite[Theorem 5.1]{Eckhard2024} since the bundle map $\ker \# \to A$ is by construction an algebroid morphism). Since cochain maps preserve the differential, the restriction of a 2-cocycle is again a 2-cocycle. Similarly, the restriction of a coboundary remains a coboundary.
\end{proof}


\section{The Lie algebroid of broken vector fields}\label{sec:algebroid}

In this section we consider vector fields on the circle $S^1$ which are everywhere continuous and are $C^\infty$-smooth away from a finite number of ``break'' points $p_1, \dots, p_n \in S^1$. We will always assume that the points $p_1, \dots, p_n$ are cyclically ordered. Formally, we define vector fields with breaks as follows.

\begin{definition}
A continuous vector field on the circle $S^1$ is said to be a vector field with breaks at $p_1, \dots, p_n \in S^1$ if, on each connected component of $S^1 \setminus \{p_1, \dots, p_n\}$, it is smooth up to the boundary; that is, it is smooth on $S^1 \setminus \{p_1, \dots, p_n\}$ and admits left and right derivatives of all orders at each point $p_1, \dots, p_n$.
\end{definition}




We denote by $\BVect_n(S^1)$ the set of all vector fields on the circle with $n$ breaks. In contrast to the space of all smooth vector fields, $\BVect_n(S^1)$ is no longer a vector space but rather a vector bundle over the manifold $\B \subset (S^1)^n$ consisting of $n$-tuples of cyclically ordered points $p_1, \dots, p_n \in S^1$. The projection $\BVect_n(S^1) \to \B$ associates to a vector field with $n$ breaks the positions of its break points.

In this section we will show that this bundle can be equipped with the structure of a Lie algebroid. The corresponding bracket generalizes the Lie bracket of smooth vector fields on the circle. 

To describe this structure, note that a section of $\BVect_n(S^1)$ is a family of vector fields
$
u(x,p)\partial_x
$
on the circle depending on a parameter $p = (p_1, \dots, p_n) \in \B$ such that, for every $p \in \B$, the vector field $u(x,p)\partial_x$ has breaks at $p_1, \dots, p_n$.

\begin{definition}[{(The Lie algebroid of broken vector fields)}]\label{def:alg-broken}
The bracket of two sections
\[
u(x,p)\partial_x, \qquad v(x,p)\partial_x
\]
of $\BVect_n(S^1)$ is defined by
\begin{gather}
\dbo u(x,p)\partial_x,\, v(x,p)\partial_x \dbc 
= \big(u(x,p)v_x(x,p) - v(x,p)u_x(x,p)\big)\partial_x \\
\quad + \sum_{i=1}^n \big(u(p_i,p)\,v_{p_i}(x,p) - v(p_i,p)\,u_{p_i}(x,p)\big)\partial_x,
\end{gather}
where subscripts denote partial derivatives.

The anchor map $\# \colon \BVect_n(S^1) \to T\B$ is defined by sending a broken vector field to its values at the break points.
\end{definition}
We use the $\dbo \,,\dbc$ notation for the bracket of sections to distinguish it from the Lie bracket of vector fields. Also note that for sections
\[
\mathbf{u} = u(x,p)\partial_x, \quad 
\mathbf{v} = v(x,p)\partial_x,
\]
the bracket in the algebroid can be rewritten as
\begin{gather}
\dbo \mathbf{u}, \mathbf{v} \dbc = [\mathbf{u}, \mathbf{v}] + L_{\#\mathbf{u}} \mathbf{v} - L_{\#\mathbf{v}} \mathbf{u},
\end{gather}
since
\[
L_{\#\mathbf{u}} = \sum_{i=1}^n u(p_i,p)\,\frac{\partial}{\partial p_i}.
\]

\begin{remark}
Note that the first term in the above formula is simply the usual Lie bracket of vector fields. By itself, however, it does not define a bracket on $\BVect_n(S^1)$, since the continuity condition for broken vector fields is not preserved under the Lie bracket. The remaining terms can be viewed as a \emph{correction} that restores the required continuity.

A similar situation arises in the study of vortex sheets in incompressible fluids \cite{IzoKh2017}. There, the algebroid consists of vector fields that are discontinuous across a hypersurface but whose normal component is continuous. Such vector fields are not closed under the Lie bracket and therefore do not form a Lie algebra. In that context, correction terms of the form 
$L_{\#\mathbf{u}} \mathbf{v} - L_{\#\mathbf{v}} \mathbf{u}$
are introduced to ensure closure under the algebroid bracket.

The setting of the present paper can be viewed as a one-dimensional analogue of vortex sheets in a compressible fluid, which explains the appearance of similar bracket structures.
\end{remark}

\begin{proposition}
The above definitions of the bracket and anchor turn $\BVect_n(S^1)$ into a Lie algebroid.
\end{proposition}

\begin{proof}
We first verify the Leibniz rule. For $u(x,p)\partial_x, v(x,p)\partial_x \in \Gamma(\BVect_n(S^1))$ and a smooth function $f(p)$, we have
\begin{gather}
\dbo u(x,p)\partial_x, f(p)\,v(x,p)\partial_x \dbc 
= f(p)\,\dbo u(x,p)\partial_x, v(x,p)\partial_x \dbc \\
+ \left( \sum_{i=1}^n u(p_i,p)\, \frac{\partial f}{\partial p_i}(p) \right) v(x,p)\,\partial_x,
\end{gather}
so the Leibniz rule holds.

Next, we check that the space of sections of $\BVect_n(S^1)$ is closed under the bracket $\dbo \,,\dbc$, and that the Jacobi identity is satisfied. To this end, to each section $u(x,p)\partial_x$ we associate the vector field
\begin{equation}
u(x,p)\,\partial_x + \sum_{i=1}^n u(p_i,p)\,\partial_{p_i}
\end{equation}
on $S^1 \times \B$. This gives a vector space isomorphism between $\Gamma(\BVect_n(S^1))$ and the space of vector fields on $S^1 \times \B$ of the form
\begin{equation}
u(x,p)\,\partial_x + \sum_{i=1}^n u(p_i,p)\,\partial_{p_i},
\end{equation}
which are everywhere continuous, smooth up to the boundary in each connected component of the complement of the hyperplanes $x = p_1, \dots, x = p_n$, and tangent to these hyperplanes. Observe that this space is closed under the usual Lie bracket of vector fields. Moreover, we have
\begin{gather}
\Big[ u(x,p)\,\partial_x + \sum_{i=1}^n u(p_i,p)\,\partial_{p_i},\,
v(x,p)\,\partial_x + \sum_{i=1}^n v(p_i,p)\,\partial_{p_i} \Big] \\
= \dbo u(x,p)\partial_x, v(x,p)\partial_x \dbc + \sum_{i=1}^n \cdots \,\partial_{p_i},
\end{gather}
so the above assignment identifies the bracket $\dbo \,,\dbc$ with the standard Lie bracket of vector fields. The result follows.
\end{proof}
\begin{remark}
In the case of a single break point $p$, this identification allows one to view the algebroid bracket as the Lie bracket on a Lie algebra of vector fields on the 2-torus. These vector fields are of the form
\[
u(x,p)\,\partial_x + u(p,p)\,\partial_p,
\]
which are tangent to the ``diagonal'' $x=p$ and whose $p$-component is independent of $x$.
\end{remark}


\section{The algebroid Gelfand-Fuchs cocycle and Virasoro algebroid}

Recall that the Gelfand--Fuchs 2-cocycle on the Lie algebra of vector fields on the circle $S^1$ can be written in several equivalent forms; see Definition~\ref{def:GF-cocycle}. 
Below we show that if one uses the form
\[
\int u_x v_{xx}\,dx
\]
followed by an additional skew-symmetrization, this yields a cocycle on the algebroid $\BVect_n(S^1)$ of broken vector fields.
\begin{theorem}\label{thm:GFA}
Let $u(x)\partial_x$ and $v(x)\partial_x$ be vector fields with breaks at points $p_1, \dots, p_n \in S^1$ (so that they belong to the same fiber of the algebroid $\BVect_n(S^1)$). Then $(\Omega_1, \dots, \Omega_n)$, where
\[
\Omega_i\!\left(u(x)\partial_x,\, v(x)\partial_x\right)
:=
\int_{p_i}^{p_{i+1}} \bigl(u_x v_{xx} - u_{xx} v_x\bigr)\,dx,
\]
defines an $n$-dimensional cocycle on $\BVect_n(S^1)$.
\end{theorem}

\begin{proof}
It suffices to show that, for each $i$, the bilinear form
\[
\Omega\!\left(u(x)\partial_x,\, v(x)\partial_x\right)
:=
\int_{p_i}^{p_{i+1}} \bigl(u_x v_{xx} - u_{xx} v_x\bigr)\,dx
\]
defines a cocycle. Let
\[
\mathbf{u} = u(x,p)\partial_x, \quad
\mathbf{v} = v(x,p)\partial_x, \quad
\mathbf{w} = w(x,p)\partial_x
\]
be three sections of $\BVect_n(S^1)$. Then
\begin{align}
&L_{\#\mathbf{u}} (\Omega(\mathbf{v}, \mathbf{w})) - \Omega([[\mathbf{u}, \mathbf{v}]], \mathbf{w}) \,+ \circlearrowleft \notag \\
&= L_{\#\mathbf{u}} (\Omega(\mathbf{v}, \mathbf{w})) - \Omega([\mathbf{u}, \mathbf{v}], \mathbf{w}) - \Omega(L_{\#\mathbf{u}}\mathbf{v}, \mathbf{w}) + \Omega(L_{\#\mathbf{v}}\mathbf{u}, \mathbf{w}) \,+ \circlearrowleft \notag \\
&= L_{\#\mathbf{u}} (\Omega(\mathbf{v}, \mathbf{w})) - \Omega([\mathbf{u}, \mathbf{v}], \mathbf{w}) - \Omega(L_{\#\mathbf{u}}\mathbf{v}, \mathbf{w}) - \Omega(\mathbf{v}, L_{\#\mathbf{u}}\mathbf{w}) \,+ \circlearrowleft \notag \\
&= (L_{\#\mathbf{u}} \Omega)(\mathbf{v}, \mathbf{w}) - \Omega([\mathbf{u}, \mathbf{v}], \mathbf{w}) \,+ \circlearrowleft.
\end{align}

Next, we compute
\begin{gather}
\Omega([\mathbf{u}, \mathbf{v}], \mathbf{w}) \,+ \circlearrowleft
= \int_{p_i}^{p_{i+1}} 
\Bigl(
(u v_x - u_x v)_x \, w_{xx} 
- (u v_x - u_x v)_{xx} \, w_x
\Bigr) dx \,+ \circlearrowleft \notag\\
= \int_{p_i}^{p_{i+1}} \Bigl[
(u v_{xx} - u_{xx} v) \, w_{xx} 
- (u_x v_{xx} - u_{xx} v_x) \, w_x - (u v_{xxx} - u_{xxx} v) \, w_x
\Bigr] dx \,+ \circlearrowleft.
\end{gather}

All terms not involving third derivatives cancel after summing over cyclic permutations. Hence,
\[
\Omega([\mathbf{u}, \mathbf{v}], \mathbf{w}) \,+ \circlearrowleft \,
= \int_{p_i}^{p_{i+1}} (u_{xxx} v - u v_{xxx})\, w_x \, dx \,+ \circlearrowleft.
\]

We now integrate by parts, moving third derivatives off $u$ and $v$. The resulting integral terms cancel under cyclic summation, leaving only boundary contributions:
\begin{align}
\Omega([\mathbf{u}, \mathbf{v}], \mathbf{w}) \,+ \circlearrowleft
&= \bigl( (u_{xx} v - u v_{xx}) w_x \bigr)\Big|_{p_i}^{p_{i+1}} \,+ \circlearrowleft \notag \\
&= \bigl( u v_x w_{xx} - u v_{xx} w_x \bigr)\Big|_{p_i}^{p_{i+1}} \,+ \circlearrowleft \notag \\
&= (L_{\#\mathbf{u}} \Omega)(\mathbf{v}, \mathbf{w}) \,+ \circlearrowleft.
\end{align}

This proves the cocycle condition.
\end{proof}

\begin{proposition}
The cocycle $\sum c_i \Omega_i$ is not a coboundary for any nonzero vector 
\[
c = (c_1, \dots, c_n) \in \mathbb{R}^n.
\]
Equivalently, $\sum c_i \Omega_i$  defines a non-trivial one-dimensional extension for any $c\neq 0$. 
\end{proposition}

\begin{proof}
Consider the isotropy subalgebra over $(p_1, \dots, p_n) \in B$. By construction, it is a direct sum of algebras 
\[
\bigoplus_{i=1}^{n} \Vect_0([p_i, p_{i+1}]).
\]

Suppose, for contradiction, that $\Omega_c := \sum c_i \Omega_i$  is a coboundary and that $c_i \neq 0$ for some~$i$. Then the restriction of $\Omega_c$ to the subalgebra $\Vect_0([p_i, p_{i+1}])$ would also be a coboundary. This, however, contradicts Proposition~\ref{NCB} applied to the interval $[p_i, p_{i+1}]$. Therefore, $\Omega_c$ cannot be a coboundary for any nonzero $c$.
\end{proof}

\begin{definition}
The $n$-dimensional central extension of $\BVect_n(S^1)$ defined by the cocycles $\Omega_i$ is called the \textit{Virasoro algebroid}.
\end{definition}

\section{Lie groupoids and their central extensions}
\subsection{Lie groupoids}
\begin{definition}
A \textit{groupoid} $G \rightrightarrows B$ consists of two sets: a set $B$ of \textit{objects} and a set $G$ of \textit{arrows}, together with the following data:
\begin{enumerate}
\item Two maps $\Src, \Trg \colon G \to B$, called the \textit{source} and \textit{target} maps.

\item A partially defined binary operation $(g,h) \mapsto gh$ on $G$, defined for all pairs $g,h \in G$ such that $\Src(g) = \Trg(h)$, satisfying:
\begin{enumerate}
\item $\Src(gh) = \Src(h)$ and $\Trg(gh) = \Trg(g)$,
\item (Associativity) $g(hk) = (gh)k$ whenever either side is defined,
\item (Identity) For each $x \in B$, there exists an element $\id_x \in G$ such that $\Src(\id_x) = \Trg(\id_x) = x$ and
\[
\id_{\Trg(g)} \cdot g = g \cdot \id_{\Src(g)} = g
\quad \text{for all } g \in G,
\]
\item (Inverse) For each $g \in G$, there exists an element $g^{-1} \in G$ such that
\[
\Src(g^{-1}) = \Trg(g), \quad \Trg(g^{-1}) = \Src(g),
\]
and
\[
g^{-1}g = \id_{\Src(g)}, \qquad gg^{-1} = \id_{\Trg(g)}.
\]
\end{enumerate}
\end{enumerate}
\end{definition}

In what follows, we will often use the term \textit{groupoid} to refer to the set of arrows $G$. If $G \rightrightarrows B$ is a groupoid, we say that $G$ is a groupoid over $B$.

\begin{definition}
A groupoid $G \rightrightarrows B$ is called a \textit{Lie groupoid} if $G$ and $B$ are smooth manifolds, the source and target maps are submersions, and the maps
\[
(g,h) \mapsto gh, \quad x \mapsto \id_x, \quad g \mapsto g^{-1}
\]
are smooth.
\end{definition}

\begin{remark}
The domain of the multiplication map $(g,h) \mapsto gh$ is
\[
G^{(2)} := \{(g,h) \in G \times G \mid \Src(g) = \Trg(h)\}.
\]
It is a submanifold of $G \times G$ because $\Src$ and $\Trg$ are submersions; hence the smoothness condition for multiplication is well defined.
\end{remark}

\begin{example}\label{groupoids}
\begin{enumerate}[label=(\alph*)]
\item Any Lie group $G$ can be regarded as a Lie groupoid over a point.

\item For any smooth manifold $B$, the set $G := B \times B$ defines a Lie groupoid over $B$, called the \textit{pair groupoid}. The source and target maps are given by
\[
\Src(x,y) = x, \qquad \Trg(x,y) = y,
\]
and the multiplication is defined by
\[
(y,z)\cdot(x,y) := (x,z).
\]
\end{enumerate}
\end{example}

\subsection{Algebroids via bisections}

\begin{definition}
The \textit{Lie algebroid $A \to B$ associated to a Lie groupoid $\mathcal{G} \rightrightarrows B$} is a vector bundle $A \to B$ whose fiber over a point $x \in B$ is the tangent space at the identity element $\id_x$ to the source fiber $\mathcal{G}_x := \Src^{-1}(x)$. The \textit{anchor map} is defined as the differential of the target map,
\[
\# = d(\Trg) : A \to TB.
\]
\end{definition}

To define a Lie bracket on sections of \(A\), we use the group of \emph{bisections} of \(\mathcal{G}\). \begin{definition}
A \emph{bisection} of a Lie groupoid $G \rightrightarrows B$ is a submanifold $X \subset G$ such that both the source and target maps, when restricted to $X$, are diffeomorphisms onto $B$. Equivalently, a bisection can be viewed as a smooth map $\phi_X : B \to G$ satisfying $\Src \circ \phi_X = \mathrm{id}_B$ and such that $\Trg \circ \phi_X$ is a diffeomorphism  $B \to B$.
\end{definition}

The set of bisections of a Lie groupoid forms a Lie group under the multiplication
\begin{equation}\label{eq:bmult}
(\phi \psi)(x) := \phi(\Trg(\psi(x))) \cdot \psi(x).
\end{equation}
The identity element is given by $x \mapsto \id_x$, and the inverse of a bisection $\phi$ is
\[
\phi^{-1}(x) = \phi\big((\Trg \circ \phi)^{-1}(x)\big)^{-1}.
\]

Unless the base is a point, the group of bisections is infinite-dimensional, even if the groupoid itself is finite-dimensional.

\begin{example}
When $\mathcal{G}$ is the pair groupoid $X \times X \rightrightarrows X$, its group of bisections can be naturally identified with the group of diffeomorphisms of $X$.
\end{example}

Since a bisection is a section of the source map, a tangent vector at the identity bisection can be identified with a section of the bundle $A$. In this way, $\Gamma(A)$ is naturally identified with the Lie algebra of the group of bisections and therefore inherits a Lie bracket.

\subsection{Cocycles and central extensions}

We now discuss \emph{groupoid cocycles} and their role in defining \emph{central extensions of Lie groupoids}.

\begin{definition}
A smooth function $\chi \colon G^{(2)} \to \mathbb{R}$ is called a \emph{groupoid cocycle} if
\begin{equation}\label{eq:gr-cocycle}
\chi(\phi, \psi \eta) + \chi(\psi, \eta)
=
\chi(\phi \psi, \eta) + \chi(\phi, \psi),
\end{equation}
for all $(\phi,\psi,\eta) \in G^{(3)}$, where
\[
G^{(3)} := \{ (\phi,\psi,\eta) \in G \times G \times G \mid 
\Src(\phi) = \Trg(\psi), \ \Src(\psi) = \Trg(\eta) \}
\]
denotes the set of composable triples.
\end{definition}

A groupoid cocycle \(\chi\) defines a {central extension} of \(G\) by \(\mathbb{R}\) as follows. Consider the manifold
\[
\widehat{G} := G \times \mathbb{R},
\]
with source and target maps
\[
\Src(\phi, t) := \Src(\phi), \qquad \Trg(\phi, t) := \Trg(\phi),
\]
and multiplication
\[
(\phi, t)  (\psi, s) := (\phi  \psi,\, t + s + \chi(\phi, \psi)),
\]
for composable \(\phi, \psi \in G\). The cocycle condition \eqref{eq:gr-cocycle} ensures that this multiplication is associative, thus giving a one-dimensional central extension of \(G\). Higher-dimensional cocycles and central extensions are defined in a similar way.

The Lie algebroid of the extended groupoid $\hat G \rightrightarrows B$ is a central extension of the algebroid $A \to B$ corresponding to $G$. The relation between the corresponding cocycles is as follows. A cocycle $\chi$ on a groupoid $G \rightrightarrows B$ defines a $C^\infty(B)$-valued cocycle on the corresponding group of bisections by the formula
\[
X(\phi, \psi)(x) := \chi\big(\phi(\Trg(\psi(x))), \psi(x)\big),
\]
where $\phi, \psi$ are bisections of $G$ and $x \in B$. Differentiating this cocycle along one-parameter families of bisections gives rise to a cocycle $\Omega$ on the corresponding Lie algebra $\Gamma(A)$. This cocycle is also $C^\infty(B)$-valued, i.e., a bilinear map
\[
\Omega \colon \Gamma(A) \times \Gamma(A) \to C^\infty(B).
\]
Moreover, $\Omega$ is $C^\infty(B)$-bilinear, i.e. a $2$-form on $A$. The groupoid cocycle condition on $\chi$ ensures that this $2$-form is a Lie algebroid cocycle. Explicitly, for $u, v \in \Gamma(A)$, one has
\begin{gather*}
\Omega(u, v)(x)
= \frac{\partial^2}{\partial t \, \partial s} \Big|_{t=s=0} 
\Big( X(\phi^t, \psi^s)(x) - X(\psi^s, \phi^t)(x) \Big) \\
= \frac{\partial^2}{\partial t \, \partial s} \Big|_{t=s=0} 
\Big( \chi\big(\phi^t(\Trg(\psi^s(x))), \psi^s(x)\big) 
- \chi\big(\psi^s(\Trg(\phi^t(x))), \phi^t(x)\big) \Big).
\end{gather*}
where $\phi^t, \psi^s$ are one-parameter families of bisections with tangent vectors $u, v$ at $t=s=0$.

\section{The groupoid of broken diffeomorphisms}\label{sect:brok-diff}

In this section, we consider {homeomorphisms of the circle} that are smooth 
everywhere except at a finite set of points. Such \emph{broken diffeomorphisms} 
do not form a group, since composition is not always defined. Instead, they form a 
{groupoid}, because one can only compose broken diffeomorphisms whose sets 
of non-smooth points match: the target of one broken diffeomorphism must coincide 
with the source of the next. We now give a formal definition. Recall that $\B \subset (S^1)^n$ is the space of $n$-tuples of cyclically ordered points $p_1, \dots, p_n \in S^1$.

\begin{definition}
The \emph{broken diffeomorphism groupoid}
\[
\mathrm{BDiff}(S^1, n) \rightrightarrows B, \quad n \in \mathbb{N},
\]
consists of homeomorphisms \(\phi: S^1 \to S^1\) that are diffeomorphisms away from \(n\) cyclically ordered points \(p = (p_1, \dots, p_n) \in S^1\). Moreover, the restriction of \(\phi\) to each closed interval \([p_j, p_{j+1}]\) is a diffeomorphism onto its image.
The points \(p_j\) are called the \emph{broken points} of \(\phi\).

The \emph{source} and \emph{target} maps are given by
\[
s(\phi) = p, \qquad t(\phi) = \phi(p).
\]

Composition is defined whenever the source of one element coincides with the target of another, and is given by the usual composition of circle homeomorphisms.
\end{definition}



\begin{remark}
One can define the structure of a Lie--Fr\'echet groupoid on $\mathrm{BDiff}(S^1, n) $ by using the framework of the vortex sheet groupoid defined in \cite{IzoKh2017}. Here we give an {\it ad hoc} definition 
of the associated structures recovering,  in particular, the Lie algebroid defined above.
\end{remark}

Let us show that the algebroid $\BVect_n(S^1)$ of {broken vector fields} defined in Section~\ref{sec:algebroid} can be regarded as an infinitesimal object corresponding to $\mathrm{BDiff}(S^1, n) $. Consider the algebroid associated with the groupoid $\mathrm{BDiff}(S^1, n) $. By definition, its fiber over \(p \in B\) is the tangent space at the identity to the source fiber of the groupoid over \(p\). The source fiber over \(p = (p_1, \ldots, p_n)\) consists of diffeomorphisms that are broken at \((p_1, \ldots, p_n)\). Considering a smooth family of such diffeomorphisms starting at the identity, the tangent vector at \(t = 0\) is precisely a broken vector field, as defined in Section~\ref{sec:algebroid}. This shows that, as a vector bundle, the Lie algebroid associated with \(\mathrm{BDiff}(S^1, n)\) coincides with \(\mathrm{BVect}(S^1, n)\).

Next, we compute the anchor map. By definition, for a curve \(\phi^t\) of broken diffeomorphisms, we have
\[
\#\!\left(\left.\frac{d}{dt}\right|_{t=0} \phi^t\right)
= \left.\frac{d}{dt}\right|_{t=0} {\Trg}(\phi^t)
= \left.\frac{d}{dt}\right|_{t=0} \phi^t(p).
\]
Thus, the anchor is given by evaluation at the points \((p_1, \ldots, p_n)\), as required.

Finally, we show that algebroid brackets agree. To this end, we first describe the  {group of bisections} of the groupoid \(\mathrm{BDiff}(S^1, n)\).

By definition, a bisection is, in particular, a section of the source map
\[
\mathrm{BDiff}(S^1, n) \to B,
\]
that is, a family $\phi(x,p)$ of broken diffeomorphisms parametrized by \(p \in B\). For each \(p \in B\), the map \(x \mapsto \phi(x,p)\) is an element of \(\mathrm{BDiff}(S^1, n)\), and its broken points are precisely \(p\). Such a family $\phi(x,p)$ defines a bisection if and only if the map
\[
p \mapsto \phi(p,p) := (\phi(p_1, p), \dots, \phi(p_n, p)).
\]
is a diffeomorphism \(B \to B\).

Any bisection  $\phi(x,p)$ can be identified with a map
\[
\Phi: S^1 \times B \to S^1 \times B,
\]
given by
\[
\Phi(x,p) = (\phi(x,p), \, \phi(p,p)).
\]\
This is a bisection when the second component is a diffeomorphism \(B \to B\).
\begin{proposition}
This defines an isomorphism between the group of bisections of the groupoid \(\mathrm{BDiff}(S^1, n)\) and the group of homeomorphisms \(\Phi\) of \(S^1 \times B\) satisfying the following conditions:

1) There exists a diffeomorphism \(f: B \to B\) such that the diagram
\[
\begin{tikzcd}
S^1 \times B \arrow[r, "\Phi"] \arrow[d, "\pi_B"'] & S^1 \times B \arrow[d, "\pi_B"] \\
B \arrow[r, "f"] & B
\end{tikzcd}
\]
commutes, where \(\pi_B(x,p) = p\).

2) \(\Phi\) maps the complement of the union of hypersurfaces \(x = p_i\) diffeomorphically to itself. Moreover, if we cut \(S^1 \times B\) along these hypersurfaces, then \(\Phi\) extends to a self-diffeomorphism of the resulting manifold with boundary \(\widetilde{S^1 \times B}\). 
\end{proposition}

\begin{proof}
Let \(\phi_p\) and \(\psi_p\) be bisections of the groupoid \(\mathrm{BDiff}(S^1, n)\). The corresponding homeomorphisms of \(S^1 \times B\) are
\[
\Phi(x,p) = (\phi(x,p), \, \phi(p,p)) \quad \text{and} \quad 
\Psi(x,p) = (\psi(x,p), \, \psi(p,p)).
\] 

Their composition \(\Psi \circ \Phi\) is
\[
(\Psi \circ \Phi)(x,p) = \Psi\big(\Phi(x,p)\big) 
= \Psi\big(\phi(x,p), \, \phi(p,p)\big) 
= \big(\psi_{\phi(p,p)}(\phi(x,p)), \, \psi_{\phi(p,p)}(\phi(p,p))\big).
\]

This exactly corresponds to the product of the bisections \(\psi\) and \(\phi\). Therefore, the group structure on bisections coincides with the composition of the associated homeomorphisms \(\Phi\) of \(S^1 \times B\). 
\end{proof}
Since the group of such \(\Phi\) can be identified with diffeomorphisms of \(\widetilde{S^1 \times B}\) that descend to \(B\), it is a Lie group whose Lie algebra consists of vector fields tangent to each boundary component \(x = p_i\). Such vector fields have the form
\[
u(x, p)\,\partial_x + \sum_{i = 1}^n u(p_i, p)\,\partial_{p_i}.
\]

It follows from the proposition that there is an isomorphism of Lie algebras between \(\Gamma(\mathrm{BVect}(S^1, n))\) and the algebra of such vector fields with the standard Lie bracket. Explicitly, the isomorphism is given by
\[
u(x, p)\,\partial_x \;\longmapsto\; u(x, p)\,\partial_x + \sum_{i = 1}^n u(p_i, p)\,\partial_{p_i}.
\]

Moreover, we have already verified that the Lie bracket on these vector fields coincides with the algebroid bracket on \(\mathrm{BVect}\). This completes the proof that the Lie algebroid of \(\mathrm{BDiff}(S^1, n)\) is \(\mathrm{BVect}(S^1, n)\).

\section{The groupoid Bott cocycle and Virasoro groupoid}
\subsection{The groupoid cocycle}
\begin{definition}
The \textit{broken Virasoro groupoid} is the central extension groupoid 
\[
\mathrm{BVir}(S^1, n) \rightrightarrows B,
\]
where 
\[
\mathrm{BVir}(S^1, n) = \mathrm{BDiff}(S^1, n) \times \mathbb{R}^n,
\]
and the multiplication is given by
\[
(\phi, t) \cdot (\psi, s) = (\phi \circ \psi, t + s + \chi(\phi, \psi))
\]
whenever \(\mathrm{src}(\phi) = \mathrm{trg}(\psi)\). Here \(\chi = (\chi_1, \ldots, \chi_n) \in \mathbb{R}^n\) is the groupoid \(2\)-cocycle defined by
\[
\chi_i(\phi, \psi) = \int_{p_i}^{p_{i + 1}} \log(\phi_x \circ \psi)\, d\log \psi_x
\]
where \(\mathrm{src}(\psi) = p = (p_1, \ldots, p_n)\), \(p_{n + 1} = p_1\).
\end{definition}

\begin{theorem}\label{thm:vir-groupoid}
The broken Virasoro groupoid defined above is indeed a groupoid, 
 i.e. the composition satisfies the associativity relation. Equivalently, \(\chi\) is a groupoid cocycle.
\end{theorem}

\begin{proof}
To demonstrate the cocycle identity, consider 
\(\phi, \psi, \eta \in \mathrm{BDiff}(S^1, n)\) 
such that \(\phi \circ \psi \circ \eta\) is defined. Rewriting the left-hand side of \eqref{eq:gr-cocycle}, we have
\begin{gather*}
\chi_i(\phi, \psi \circ \eta) + \chi_i(\psi, \eta) 
 \\ = \int_{p_i}^{p_{i+1}} 
      \log(\phi_x \circ \psi \circ \eta)\, d\log((\psi \circ \eta)_x) 
 + \int_{p_i}^{p_{i+1}} \log(\psi_x \circ \eta)\, d\log \eta_x.
\end{gather*}

Applying the chain rule \((\psi \circ \eta)_x = (\psi_x \circ \eta) \cdot \eta_x\), this further rewrites as
\begin{gather*}
\int_{p_i}^{p_{i+1}} \log(\phi_x \circ \psi \circ \eta)\, d\log(\psi_x \circ \eta) + \int_{p_i}^{p_{i+1}} \log(\phi_x \circ \psi \circ \eta)\, d\log \eta_x  
 + \int_{p_i}^{p_{i+1}} \log(\psi_x \circ \eta)\, d\log \eta_x \\ = \int_{\eta(p_i)}^{\eta(p_{i+1})} \log(\phi_x \circ \psi)\, d\log \psi_x  
 + \int_{p_i}^{p_{i+1}} \log((\phi \circ \psi)_x \circ \eta)\, d\log \eta_x \\ = \chi_i(\phi, \psi) + \chi_i(\phi \circ \psi, \eta).\qedhere
\end{gather*}
\end{proof}
\begin{remark}
The Virasoro-Bott cocycle on the group $\Diff(S^1)$ is usually written in the form
\[
\int_{S^1} \log(\phi \circ \psi)_x \, d\log \psi_x.
\]

The relation between the group and groupoid versions is as follows. 
The groupoid cocycle $\sum \chi_i$ corresponding to the whole circle is
\[
\sum \chi_i=\sum_{i=1}^n \int_{p_i}^{p_{i+1}} \log(\phi_x \circ \psi)\, d\log \psi_x
= \int_{S^1} \log(\phi_x \circ \psi)\, d\log \psi_x.
\]

Using the chain rule \(\phi_x \circ \psi = (\phi \circ \psi)_x \cdot (\psi_x)^{-1}\), this rewrites as
\begin{gather}
\sum_{i=1}^n \int_{p_i}^{p_{i+1}} \log(\phi \circ \psi)_x\, d\log \psi_x 
- \sum_{i=1}^n \int_{p_i}^{p_{i+1}} \log \psi_x\, d\log \psi_x \\
= \int_{S^1} \log(\phi \circ \psi)_x\, d\log \psi_x 
- \frac{1}{2} \sum_{i=1}^n \bigl(\log^2 \psi_x(p_{i+1}^-) - \log^2 \psi_x(p_i^+)\bigr).
\end{gather}

Here \(f(p_i^\pm)\) denotes the right and left limits at \(p_i\), respectively. Rewriting this further, we obtain
\begin{gather}
\int_{S^1} \log(\phi \circ \psi)_x\, d\log \psi_x 
+ \frac{1}{2} \sum_{i=1}^n \bigl(\log^2 \psi_x(p_i^+) - \log^2 \psi_x(p_i^-)\bigr).
\end{gather}

For a diffeomorphism \(\psi\) without breaks of the derivative \(\psi_x\), the non-integral terms vanish, and we recover the usual Bott cocycle.
\end{remark}

\subsection{From groupoid to algebroid cocycle}

\begin{theorem}\label{prop:gr-alg-cocycle}

The algebroid cocycle corresponding to the groupoid cocycle $\chi_i$ coincides with 
the  Gelfand-Fuchs cocycle
$$\Omega_i = \int_{p_i}^{p_{i+1}} \bigl(u_x v_{xx} - u_{xx} v_x\bigr)\,dx.$$
\end{theorem}

\begin{proof}
Let $u(x,p)\,\partial_x$ and $v(x,p)\,\partial_x$ be sections of $\mathrm{BVect}(S^1,n)$, viewed as tangent vectors to the group of bisections. Let $\phi^t(\cdot,p)$ and $\psi^s(\cdot,p)$ be curves in the group of bisections such that
\[
\left.\frac{\partial}{\partial t}\right|_{t=0} \phi^t(x,p) = u(x,p),
\qquad
\left.\frac{\partial}{\partial s}\right|_{s=0} \psi^s(x,p) = v(x,p).
\]
The algebroid cocycle is given by
\begin{gather*}
\Omega_i(u,v)(p)
=
\frac{\partial^2}{\partial t \, \partial s}\Big|_{t=s=0}
\Big(
\chi_i(\phi^t(\cdot, \psi^s(p,p)), \psi^s(\cdot,p))
-
\chi_i(\psi^s(\cdot, \phi^t(p,p)), \phi^t(\cdot,p))
\Big).
\end{gather*}
We first compute
\begin{gather*}
\frac{\partial^2}{\partial t \, \partial s}\Big|_{t=s=0}
\chi_i(\phi^t(\cdot, \psi^s(p,p)), \psi^s(\cdot,p)) \\
=
\frac{\partial^2}{\partial t \, \partial s}\Big|_{t=s=0}
\int_{p_i}^{p_{i+1}}
\big(\log(\phi^t(x,\psi^s(p,p)))_x \circ \psi^s(x,p)\big)\,
d \log \psi^s(x,p)_x.
\end{gather*}
Since $\psi^0(x,p)=x$, we have
\[
d \log \psi^s(x,p)_x \Big|_{s=0} =\,\, 0.
\]
Thus,
\begin{gather*}
\frac{\partial^2}{\partial t \, \partial s}\Big|_{t=s=0}
\Big(
\big(\log(\phi^t(x,\psi^s(p,p)))_x \circ \psi^s(x,p)\big)\,
d \log \psi^s(x,p)_x
\Big) \\
=
\frac{\partial}{\partial t}\Big|_{t=0}
\Big(
\log(\phi^t(x,p))_x
\Big)
\cdot
\frac{\partial}{\partial s}\Big|_{s=0} \Big( d 
\log \psi^s(x,p)_x \Big)
 \\
\vphantom{\frac{\partial}{\partial s}\Big|_{s=0} }=u_x(x,p)\, d v_x(x,p).
\end{gather*}
Hence,
\begin{gather*}
\frac{\partial^2}{\partial t \, \partial s}\Big|_{t=s=0}
\chi_i(\phi^t(\cdot, \psi^s(p,p)), \psi^s(\cdot,p))
=
\int_{p_i}^{p_{i+1}} u_x\, d v_x.
\end{gather*}
Similarly,
\begin{gather*}
\frac{\partial^2}{\partial t \, \partial s}\Big|_{t=s=0}
\chi_i(\psi^s(\cdot, \phi^t(p,p)), \phi^t(\cdot,p))
=
\int_{p_i}^{p_{i+1}} v_x\, d u_x.
\end{gather*}
Subtracting, we obtain
\begin{gather*}
\Omega_i(u,v)(p)
=
\int_{p_i}^{p_{i+1}}
\big(u_x\, d v_x - v_x\, d u_x\big). \qedhere
\end{gather*}
\end{proof}

\begin{corollary}
The groupoid cocycle $\chi(\phi, \psi)$ is nontrivial, i.e. is not a coboundary. 
\end{corollary}
\begin{proof}
This follows from non-triviality of the algebroid cocycle.
\end{proof}

\section{Diffeomorphisms between intervals and their central extension}\label{sec:interval}


Similarly, one may define the \emph{segment diffeomorphism groupoid}
\[
\mathrm{BDiff}(I) \rightrightarrows P,
\]
whose morphisms consist of orientation-preserving diffeomorphisms
\(\phi: [a,b] \to \mathbb{R}\)
from any interval \([a,b]\) onto its image \([\phi(a), \phi(b)]\), equipped with the natural composition of maps wherever defined. The base of this groupoid is given by
\[
P = \{(x,y) \in \mathbb{R}^2 \mid x < y\},
\]
the open half-plane. Its associated algebroid consists of smooth vector fields defined on intervals, without any requirement that they vanish at the endpoints. The fiber over \((a,b) \in P\) is the space of vector fields with domain \([a,b]\). The anchor map assigns to such a vector field its values at the endpoints, while the Lie bracket is defined in the same manner as for the algebroid of broken vector fields on the circle; see Definition~\ref{def:alg-broken}.

\begin{proposition}\label{prop:interval}
The segment diffeomorphism groupoid has a one-dimensional central extension given by the cocycle
\[
\chi(\phi, \psi) = \int_{a}^{b} \log(\phi_x \circ \psi)\, d\log (\psi_x)
\]
where $\Src\, \psi = [a,b]$. The corresponding algebroid cocycle is
\[
\Omega\!\left(u(x)\partial_x,\, v(x)\partial_x\right)
:=
\int_{a}^{b} \bigl(u_x v_{xx} - u_{xx} v_x\bigr)\,dx
\]
\end{proposition}

The proof of the cocycle identities proceeds in the same manner as for broken vector fields on the circle.



\bibliographystyle{plain}
\bibliography{bibliography}

\end{document}